\documentclass[11pt,a4paper]{article}
\usepackage{amsfonts,amsgen,amstext,amsbsy,amsopn,amsfonts,amssymb,amscd}
\usepackage[leqno]{amsmath}
\usepackage[amsmath,amsthm,thmmarks]{ntheorem}
\usepackage{epsf,epsfig}
\usepackage{float}
\usepackage{dsfont}
\usepackage{ebezier,eepic}
\usepackage{color}
\usepackage{tikz}
\usepackage{multirow}
\usepackage{mathrsfs}
\usepackage{graphicx}
\usepackage{subfigure}
\newtheorem{thm}{Theorem}[section]

\newtheorem{prop}[thm]{Proposition}
\newtheorem{prob}[thm]{Problem}
\newtheorem{lem}[thm]{Lemma}
\newtheorem{example}[thm]{Example}
\newtheorem{false statement}{False statement}
\newtheorem{cor}[thm]{Corollary}

\theoremstyle{definition}
\newtheorem{defn}[thm]{Definition}
\newtheorem{claim}[thm]{Claim}

\newtheorem{conj}[thm]{Conjecture}

\makeatletter \@addtoreset{equation}{section}

\def\hh{\mathcal{H}}

\def\hl{\mathcal{L}}
\def\hht{\mathcal{T}}

\def\hf{\mathcal{F}}
\def\hg{\mathcal{G}}

\def\ha{\mathcal{A}}
\def\hb{\mathcal{B}}
\def\hs{\mathcal{S}}

\def\hc{\mathcal{C}}

\begin{document}
\title{\bf\Large On the largest degrees in intersecting hypergraphs}
\date{}
\author{Peter Frankl$^1$, Jian Wang$^2$\\[10pt]
$^{1}$R\'{e}nyi Institute, Budapest, Hungary\\[6pt]
$^{2}$Department of Mathematics\\
Taiyuan University of Technology\\
Taiyuan 030024, P. R. China\\[6pt]
E-mail:  $^1$frankl.peter@renyi.hu, $^2$wangjian01@tyut.edu.cn
}
\maketitle

\begin{abstract}
Let $\binom{[n]}{k}$ denote the collection of all $k$-subsets of the standard $n$-set $[n]=\{1,2,\ldots,n\}$. Let $n>2k$ and let $\mathcal{F}\subset \binom{[n]}{k}$ be an {\it intersecting} $k$-graph, i.e., $F\cap F'\neq \emptyset$ for all $F,F'\in \mathcal{F}$. The number of edges $F\in \mathcal{F}$ containing $x\in [n]$ is called the {\it degree} of $x$. Assume that $d_1\geq d_2\geq \ldots\geq d_n$ are the degrees of $\mathcal{F}$ in decreasing order. An important result of Huang and Zhao states that for $n>2k$ the minimum degree $d_n$ is at most $\binom{n-2}{k-2}$. For $n\geq 6k-9$ we strengthen this result by showing $d_{2k+1}\leq \binom{n-2}{k-2}$.  As to the second and third largest degrees we prove the best possible bound $d_3\leq d_2\leq \binom{n-2}{k-2}+\binom{n-3}{k-2}$ for $n>2k$. Several more best possible results of a similar nature are established.

\vspace{6pt}

\vspace{6pt}
\end{abstract}

\section{Introduction}

Let $[n]$ denote the standard $n$-element set $\{1,2,\ldots,n\}$ and $2^{[n]}$ its power set. Let $\binom{[n]}{k}$ be the family of all $k$-subsets of $[n]$. A family $\hf\subset \binom{[n]}{k}$ is called  {\it $t$-intersecting} if $|F\cap F'|\geq t$ for all  $F,F'\in \hf$ (if $t=1$ simply {\it intersecting}). To avoid trivialities, without further mention we are going to assume $n\geq 2k$ in the case of intersecting families and $n> 2k-t$ in the case of $t$-intersecting families. For $x\in [n]$, define
\[
\hf(x) = \{F\setminus \{x\}\colon x\in F\in\hf\}\mbox{ and } \hf(\bar{x}) =\{F\colon x\notin F\in \hf\}.
\]
Note that $|\hf(x)|+|\hf(\bar{x})| =|\hf|$.

\begin{defn}
Suppose that $x_1,\ldots,x_n$ is a permutation of $1,2,\ldots,n$ satisfying $|\hf(x_1)|\geq \ldots\geq |\hf(x_n)|$. Set $d_i(\hf) =|\hf(x_i)|$, i.e., $d_i(\hf)$ is the $i$th largest degree of $\hf$. When $\hf$ is clear from the context we sometimes write $d_i$ for short.
\end{defn}

In the literature $d_1(\hf)$ the largest degree is often denoted by $\Delta(\hf)$ and $d_n(\hf)$ the minimum degree by $\delta(\hf)$.

Let us recall some results concerning these quantities in intersecting $k$-graphs, $\hf\subset \binom{[n]}{k}$.

\vspace{6pt}
{\noindent\bf Exact Erd\H{o}s-Ko-Rado Theorem (\cite{ekr},  \cite{F78}, \cite{W84}).} Suppose that $n\geq (t+1)(k-t+1)$ and $\hf\subset \binom{[n]}{k}$ is $t$-intersecting. Then
\begin{align}\label{ineq-ekr}
|\hf| \leq \binom{n-t}{k-t}.
\end{align}
\vspace{6pt}

The Hilton-Milner Theorem is a strong stability result for the $t=1$ case of the Erd\H{o}s-Ko-Rado Theorem. We say that $\hf\subset \binom{[n]}{k}$ is {\it non-trivial} intersecting if $\cap \{F\colon F\in \hf\} =\emptyset$.

\vspace{6pt}
{\noindent\bf Hilton-Milner Theorem (\cite{HM67}).} If $n> 2k$ and $\hf\subset \binom{[n]}{k}$ is non-trivial intersecting, then
\begin{align}\label{ineq-nontrival}
|\hf| \leq \binom{n-1}{k-1}- \binom{n-k-1}{k-1} +1.
\end{align}
\vspace{6pt}

\begin{cor}
Suppose that $\hf\subset \binom{[n]}{k}$ is intersecting and $d_1(\hf)<|\hf|$. Then
\[
d_1(\hf)\leq \binom{n-1}{k-1}-\binom{n-k-1}{k-1}.
\]
\end{cor}

Note that for the full star $\hs =\{S\in \binom{[n]}{k}\colon 1\in S\}$, $d_i(\hs) =\binom{n-2}{k-2}$ for $2\leq i\leq n$. Huang and Zhao gave a beautiful algebraic proof showing that $d_n(\hf) \leq \binom{n-2}{k-2}$ is always true.

\begin{thm}[\cite{HZ}]\label{thm-hz}
For $n>2k$, if $\hf\subset \binom{[n]}{k}$ is intersecting, then
\[
d_n(\hf) \leq \binom{n-2}{k-2}.
\]
\end{thm}

There is a natural sequence of intersecting families containing the Hilton-Milner Family, $\hh_k$.

\begin{defn}
For $2\leq \ell\leq k$, define
\[
\hh_\ell = \hh_\ell(n,k) =\left\{H\in \binom{[n]}{k}\colon 1\in H,\ H\cap [2,\ell+1]\neq \emptyset\right\}\cup \left\{H\in \binom{[n]}{k}\colon [2,\ell+1] \subset H\right\}.
\]
\end{defn}

It is easy to verify that
\begin{align}\label{ineq-di}
&d_i(\hh_\ell) =\binom{n-2}{k-2}+\binom{n-\ell-1}{k-\ell} \mbox{ for } 2\leq i\leq \ell+1,\\[3pt]
& |\hh_\ell| =\binom{n-1}{k-1} - \binom{n-\ell-1}{k-1}+\binom{n-\ell-1}{k-\ell}.
\end{align}

A by now classical result of the first author relates the maximum degree and the size of an intersecting family.

\begin{thm}[\cite{F87-2}]\label{thm-f87}
Suppose that $\hf\subset \binom{[n]}{k}$ is intersecting, $n>2k$, $2\leq \ell\leq k$ and $d_1(\hf)\leq d_1(\hh_\ell)=\binom{n-1}{k-1}-\binom{n-\ell-1}{k-1}$. Then
\[
|\hf| \leq |\hh_\ell|.
\]
Moreover, in case of equality either $\hf$ is isomorphic to $\hh_\ell$ or $\ell=3$ and $\hf\cong \hh_2$.
\end{thm}

Since $\hh_2 =\{H\in \binom{[n]}{k}\colon |H\cap [3]|\geq 2\}$, it is often called the {\it triangle family}.

In this paper, we are mainly concerned with the $i$th degree of an intersecting family $\hf\subset \binom{[n]}{k}$. For $k=2$, there are only two types of  intersecting families, the stars and the triangle. Therefore we always assume $k\geq  3$.

Our first result shows that the maximum of $d_2$ and $d_3$
 are given by the triangle family.   The surprisingly short proof relies on the idea of Daykin \cite{daykin}.

 \begin{thm}\label{thm-main-1}
 Suppose that $\hf\subset \binom{[n]}{k}$ is intersecting, $n>2k$. Then
 \begin{align}
d_2(\hf) \leq \binom{n-2}{k-2}+\binom{n-3}{k-2}.
 \end{align}
 Moreover, if equality holds for $d_2(\hf)$ then $\hf$ is isomorphic to $\hh_2$.
 \end{thm}

 For $n\geq 6k-9$, we determine $d_{2k+1}(\hf)$ for an arbitrary intersecting family $\hf\subset \binom{[n]}{k}$.

\begin{thm}\label{thm-main-4}
Let $\hf\subset \binom{[n]}{k}$ be an intersecting family. Then for $n\geq 6k-9$,
\[
d_{2k+1}(\hf) \leq \binom{n-2}{k-2}.
\]
\end{thm}

Note that in the above range Theorem \ref{thm-main-4} is a considerable strengthening of the Huang-Zhao Theorem (Theorem \ref{thm-hz}).

For $n\geq \lceil\frac{8k}{3}\rceil$, we prove that $d_{\lceil\frac{8k}{3}\rceil}(\hf) \leq \binom{n-2}{k-2}$.

\begin{thm}\label{thm-main-6}
Let $\hf\subset \binom{[n]}{k}$ be an intersecting family. Then for $n\geq \lceil\frac{8k}{3}\rceil$,
\[
d_{\lceil\frac{8k}{3}\rceil}(\hf) \leq \binom{n-2}{k-2}.
\]
\end{thm}

For $n\geq \binom{t+2}{2}k^2$, we can prove the stronger bound $d_{k+2}(\hf)\leq \binom{n-t-1}{k-t-1}$ for any $t$-intersecting family $\hf\subset\binom{[n]}{k}$.

\begin{thm}\label{thm-main-5}
Suppose that $\hf\subset\binom{[n]}{k}$ is $t$-intersecting and $n\geq \binom{t+2}{2}k^2$, $k>t\geq 1$. Then
\begin{align}\label{ineq-5.5}
d_{k+2}(\hf)\leq \binom{n-t-1}{k-t-1}.
\end{align}
\end{thm}

Define
\[
\hh(n,k,t) =\left\{H\in \binom{[n]}{k}\colon [t]\subset H, H\cap [t+1,k+1]\neq \emptyset\right\}\cup \left\{[k+1]\setminus \{j\}\colon 1\leq j\leq t\right\}.
\]
It is easy to check that
\[
d_{t+1}(\hh(n,k,t) ) =\ldots=d_{k+1}(\hh(n,k,t) )=\binom{n-t-1}{k-t-1}+t.
\]
This shows that in a certain sense  \eqref{ineq-5.5} is best possible.

 Concerning the general case $d_{\ell}$ for  $4\leq \ell\leq k+1$ we have the following.

 \begin{prob}
 Determine or estimate the minimum value $n_0(k,\ell)$ such that for $n>n_0(k,\ell)$ the $(\ell+1)$th largest degree in an intersecting
 $k$-graph $\hf$ on $n$ vertices satisfies
 \begin{align}\label{ineq-d-ell}
 d_{\ell+1}(\hf)\leq \binom{n-2}{k-2} +\binom{n-\ell-1}{k-\ell}.
 \end{align}
 \end{prob}

 The next two results  establish \eqref{ineq-d-ell} under certain constraints.

\begin{thm}\label{thm-main-2}
Suppose that $\hf\subset \binom{[n]}{k}$ is intersecting, $n\geq 6k$. Then
 \begin{align*}
 d_4(\hf)\leq \binom{n-2}{k-2}+\binom{n-4}{k-3}.
 \end{align*}
\end{thm}

For the case $\ell \geq 4$ let us prove a slightly weaker bound.

\begin{thm}\label{thm-main-3}
Suppose that $\hf\subset \binom{[n]}{k}$ is intersecting, $4\leq \ell \leq k$ and $n>2\ell^2k$. Then
\begin{align}\label{ineq-d-ell2}
d_{\ell+1}(\hf)\leq \binom{n-2}{k-2} +\binom{n-\ell-1}{k-\ell}.
\end{align}
\end{thm}

\begin{example}
For $n\geq k\geq r\geq 1$ define
\[
\hl_r =\left\{F\in \binom{[n]}{k}\colon |F\cap[2r-1]|\geq r\right\}.
\]
Clearly $\hl_r$ is intersecting and
\[
d_i(\hl_r)= \sum_{r-1\leq j\leq 2r-2} \binom{2r-2}{j} \binom{n-2r+1}{k-j-1} \mbox{ for }i=1,2,\ldots,2r-1.
\]
\end{example}
Recall that
\[
d_{4}(\hh_{3}) =\binom{n-2}{k-2}+\binom{n-4}{k-3}= \binom{n-5}{k-2}+ 4\binom{n-5}{k-3}+ 4\binom{n-5}{k-4}+ \binom{n-5}{k-5}.
\]
Then for $n<3k-2$,
\[
d_{4}(\hl_3)-d_{4}(\hh_{3}) = 2\binom{n-5}{k-3}-\binom{n-5}{k-2} =\left(\frac{2(k-2)}{n-k-2}-1\right) \binom{n-5}{k-2}>0.
\]
Thus Theorem \ref{thm-main-2} does not hold for $n<3k-2$.

Note also that
\begin{align*}
d_{5}(\hh_{4}) &=\binom{n-2}{k-2}+\binom{n-5}{k-4}=\binom{n-5}{k-2}+3\binom{n-5}{k-3}+4\binom{n-5}{k-4}+\binom{n-5}{k-5}.
\end{align*}
Then for $n<4k-4$,
\begin{align*}
d_{5}(\hl_3)-d_{5}(\hh_{4})=3\binom{n-5}{k-3}-\binom{n-5}{k-2} =\left(\frac{3(k-2)}{n-k-2}-1\right) \binom{n-5}{k-2}>0.
\end{align*}
Thus $d_5(\hl_3)>d_5(\hh_{4})$ for $n<4k-4$.

\section{Preliminaries and the proof of Theorem \ref{thm-main-1}}

For $\hf\subset \binom{[n]}{k}$, let us define the {\it  $\ell$th shadow of $\hf$} as
\[
\partial^{(\ell)} \hf =\left\{E\in\binom{[n]}{k-\ell}\colon \mbox{there exists } F
\in \hf\mbox{ such that } E\subset F \right\}.
\]
The quintessential Kruskal-Katona Theorem \cite{Kruskal, Katona66} gives the best possible lower bounds on $|\partial^{(\ell)} \hf|$ for given size of $\hf$.

Recall the {\it co-lexicographic order} $A <_{C} B$ for $A,B\in \binom{[n]}{k}$ defined by, $A<_C B$ if and only if $\max\{i\colon i\in A\setminus B\}<\max\{i\colon i\in B\setminus A\}$. Let $\hc(n,k,m)$  be the collection of the first $m$ sets $A\in \binom{[n]}{k}$ in the co-lexicographic order.

\vspace{6pt}
{\noindent\bf  Kruskal--Katona Theorem (\cite{Kruskal, Katona66}).}
If $\hf\subset \binom{[n]}{k}$, $|\hf|= m$, then
\begin{align}\label{ineq-kk-0}
|\partial^{(\ell)} \hf|\geq |\partial^{(\ell)} \hc(n,k,m)|.
\end{align}
\vspace{2pt}

 Recall the {\it lexicographic order} $A <_{L} B$ for $A,B\in \binom{[n]}{k}$ defined by, $A<_L B$ if and only if $\min\{i\colon i\in A\setminus B\}<\min\{i\colon i\in B\setminus A\}$. For $n>k>0$ and $\binom{n}{k}\geq m>0$, let $\hl(n,k,m)$ denote the collection of the first $m$ sets $A\in \binom{[n]}{k}$ in the lexicographic order.

We also need the following reformulation of the Kruskal--Katona Theorem, which was first appeared in  \cite{Hilton}. Two families $\ha$, $\hb$ are called {\it cross-intersecting} if $A\cap B\neq \emptyset$ for all $A\in \ha$, $B\in \hb$.

\vspace{6pt}
{\noindent\bf Kruskal--Katona Theorem (\cite{Kruskal, Katona66, Hilton}).}
 Let $n,a,b$ be positive integers, $n\geq a+b$. Suppose that $\ha\subset \binom{[n]}{a}$ and $\hb\subset \binom{[n]}{b}$ are cross-intersecting. Then $\hl(n,a,|\ha|)$ and $\hl(n,b,|\hb|)$ are cross-intersecting as well.
\vspace{6pt}

For $A\subset B\subset [n]$, define
\[
\hf(A,B) =\left\{F\setminus B\colon F\cap B=A\right\}.
\]
We use $\hf(A)$ to denote $\hf(A,A)$.
We also use $\hf(i,\bar{j})$ to denote $\hf(\{i\},\{i,j\})$,  $\hf(i,j)$ to denote $\hf(\{i,j\}, \{i,j\})$ and $\hf(\bar{i},\bar{j})$ to denote $\hf(\emptyset, \{i,j\})$.

Let us recall the following corollaries of the Kruskal--Katona Theorem.

\begin{cor}\label{cor-1}
Suppose that $n\geq 2k+2$. Let $\ha\subset \binom{[n]}{k}$ and $\hb\subset \binom{[n]}{k+2}$ be cross-intersecting. If $|\ha|\geq \binom{n-1}{k-1}+\binom{n-2}{k-2}$ then
\[
|\ha|+|\hb| \leq \binom{n}{k}.
\]
\end{cor}

\begin{proof}
By the Kruskal--Katona Theorem, we may assume that  $\ha=\hl(n,k,|\ha|)$ and $\hb=\hl(n,k+2,|\hb|)$. By  $|\ha|\geq \binom{n-1}{k-1}+\binom{n-2}{k-2}$, we infer that $|\ha|=\binom{n-1}{k-1}+\binom{n-2}{k-1}+|\ha(\bar{1},\bar{2})|$ and $|\hb|=|\hb(1,2)|$. Since $\ha(\bar{1},\bar{2})$ and $\hb(1,2)$ are cross-intersecting,
\[
|\ha(\bar{1},\bar{2})|+|\hb(1,2)|\leq \binom{n-2}{k}.
\]
Thus $|\ha|+|\hb| \leq \binom{n}{k}$ follows.
\end{proof}

\begin{cor}\label{cor-hm}
Suppose $n\geq a+b$. Let $\ha\subset \binom{[n]}{a}$ and $\hb\subset \binom{[n]}{b}$ be cross-intersecting families. If $|\ha| \geq \binom{n-2}{a-2}+\binom{n-3}{a-2}+\ldots+\binom{n-d}{a-2}$ for some $2\leq d\leq b+1$, then
\[
|\hb| \leq \binom{n-1}{b-1}+\binom{n-d}{b-d+1}.
\]
Moreover, if $|\ha| \geq \binom{n-\ell}{a-\ell}$ for some $1\leq \ell\leq a$, then
\[
|\hb| \leq \binom{n}{b}-\binom{n-\ell}{b}.
\]
\end{cor}

\begin{proof}
By the Kruskal--Katona Theorem, we may assume $\ha=\hl(n,a,|\ha|)$ and  $\hb=\hl(n,b,|\hb|)$. If  $|\ha| \geq \binom{n-2}{a-2}+\binom{n-3}{a-2}+\ldots+\binom{n-d}{a-2}$, then
\[
\left\{F\in \binom{[n]}{a}\colon \{1, i\} \subset F \mbox{ for some }i\in[2,d]\right\}\subset \ha.
\]
As $\ha$, $\hb$ are cross-intersecting,
\[
\hb\subset \left\{F\in \binom{[n]}{b}\colon 1\in  F \mbox{ or } [2,d]\subset F\right\}.
\]
Consequently,
\[
|\hb| \leq \binom{n-1}{b-1}+\binom{n-d}{b-d+1}.
\]

If  $|\ha| \geq \binom{n-\ell}{a-\ell}$, then
\[
\left\{F\in \binom{[n]}{a}\colon [\ell] \subset F\right\}\subset \ha.
\]
As $\ha$, $\hb$ are cross-intersecting,
\[
\hb\subset \left\{F\in \binom{[n]}{b}\colon F\cap  [\ell]\neq \emptyset\right\}.
\]
Thus,
\[
|\hb| \leq \binom{n}{b}-\binom{n-\ell}{b}.
\]
\end{proof}

Daykin \cite{daykin} showed that the $t=1$ case of Erd\H{o}s-Ko-Rado follows from the Kruskal-Katona Theorem. The following result can be extracted from his proof.

\vspace{3pt}
{\noindent\bf  Daykin Theorem (\cite{daykin}).}
Let $n>2k$, $\ha, \hb\subset \binom{[n]}{k}$. If $\ha$ and $\hb$ are  cross-intersecting then
\[
\min\{|\ha|,|\hb|\} \leq \binom{n-1}{k-1}.
\]
In case of equality $\ha=\hb=\{F\in \binom{[n]}{k}\colon x\in F\}$ for some $x\in [n]$.
\vspace{3pt}

\begin{proof}
Suppose that $|\ha|\geq |\hb|\geq \binom{n-1}{k-1}$. We need to show  $|\ha|=\binom{n-1}{k-1}$. Assume indirectly $|\ha|>\binom{n-1}{k-1}$ and fix $\ha_0\subset \ha$, $|\ha_0|=\binom{n-1}{k-1}+1$. The initial segment $\hl(n,k,\binom{n-1}{k-1}+1)$ of the lexicographic order consists of the full  star of 1 and the set $[2,k+1]$. By the Kruskal-Katona Theorem  every member of $\hl(n,k,|\hb|)$ intersects each of these sets. Hence $|\hb|\leq \binom{n-1}{k-1}-\binom{n-k-1}{k-1}<\binom{n-1}{k-1}$, the desired contradiction.

Using the uniqueness part of the Kruskal-Katona Theorem in the case $|\ha|=\binom{n-1}{k-1}$ (cf. e.g. \cite{F84} or \cite{FG}), it follows that $|\ha|=|\hb|=\binom{n-1}{k-1}$ holds iff $\ha$ and $\hb$ are identical full stars.
\end{proof}

\begin{proof}[Proof of Theorem \ref{thm-main-1}]
Let $\hf\subset \binom{[n]}{k}$ be an intersecting family with $d_1(\hf)=|\hf(1)|$, $d_2(\hf)=|\hf(2)|$.
Note that
\[
d_1(\hf) =|\hf(1,2)|+|\hf(1,\bar{2})|,\  d_2(\hf) =|\hf(1,2)|+|\hf(\bar{1},2)|.
\]
Hence $|\hf(\bar{1},2)|\leq |\hf(1,\bar{2})|$. Since $\hf$ is intersecting, $\hf(\bar{1},2)$ and $\hf(1,\bar{2})$ are cross-intersecting. By the Daykin Theorem, we get
\[
|\hf(\bar{1},2)|=\min\{|\hf(\bar{1},2)|,|\hf(1,\bar{2})|\} \leq \binom{(n-2)-1}{(k-1)-1}.
\]
Thus  $d_2(\hf)\leq \binom{n-2}{k-2}+\binom{n-3}{k-2}$ follows. In case of equality, $|\hf(1,2)|=\binom{n-2}{k-2}$ and $\hf(\bar{1},2)$, $\hf(1,\bar{2})$ are identical full stars. Without loss of generality assume that $\hf(\bar{1},2)$, $\hf(1,\bar{2})$ are both  full stars of center $3$. Then $\hf$ is isomorphic to $\hh_2$ and $d_1(\hf) =d_2(\hf) =d_3(\hf)$ follows.
\end{proof}

\section{Saturation, minimal transversals and the shifted case}

When considering the maximum of the $i$th largest degree of a $t$-intersecting family, we may always assume that $\hf$ is saturated. That is, the addition of any new $k$-sets to $\hf$ would destroy the $t$-intersecting property.

Let $\binom{[n]}{\leq k}$ denote the collection of all subsets of $[n]$ with size at most $k$.  For $\hg\subset \binom{[n]}{\leq k}$, define
\[
\langle \hg \rangle = \left\{F\in \binom{[n]}{k}\colon \mbox{ there exists }G\in \hg \mbox{ such that }G\subset F \right\}.
\]

For a $t$-intersecting family $\hf$, define the {\it $t$-transversal family}
\[
\hht_t(\hf) := \left\{T\subset [n]\colon |T|\leq k,\ |T\cap F|\geq t \mbox{ for all } F\in \hf \right\}.
\]
Define the {\it basis} $\hb_t(\hf)$ as the family of all minimal members (for containment) of $\hht_t(\hf)$, that is , the collection of all {\it minimal $t$-transversals}.
The {\it $t$-covering number} $\tau_t(\hf)$ is defined by $\tau_t(\hf) =\min \{ |T|\colon T\in \hht_t(\hf)\}$. For $t=1$, we often write $\hht(\hf)$, $\hb(\hf)$ and $\tau(\hf)$ for $\hht_1(\hf)$, $\hb_1(\hf)$ and $\tau_1(\hf)$, respectively. The 1-covering number is simply called the {\it covering number}.

\begin{lem}[\cite{FW22}]\label{lem-3.3}
Suppose that $\hf\subset \binom{[n]}{k}$ is a saturated $t$-intersecting family and $\hb=\hb_t(\hf)$.
Then  $\hb$ is $t$-intersecting with $\langle \hb \rangle = \hf$.
\end{lem}

Let us recall the shifting method, which was invented by Erd\H{o}s, Ko and Rado \cite{ekr}. For $\hf\subset \binom{[n]}{k}$ and $1\leq i<j\leq n$, define the shift
$$S_{ij}(\hf)=\left\{S_{ij}(F)\colon F\in\hf\right\},$$
where
$$S_{ij}(F)=\left\{
                \begin{array}{ll}
                 F':= (F\setminus\{j\})\cup\{i\}, & \mbox{ if } j\in F, i\notin F \text{ and } F'\notin \hf; \\[5pt]
                  F, & \hbox{otherwise.}
                \end{array}
              \right.
$$
It is well known (cf. \cite{F87}) that shifting preserves the  $t$-intersecting property.

Let us define the {\it shifting partial order} $\prec$. For two $k$-sets $A$ and $B$ where $A=\{a_1,\ldots,a_k\}$, $a_1<\ldots<a_k$ and $B=\{b_1,\ldots,b_k\}$, $b_1<\ldots<b_k$ we say that $A$ precedes $B$ and denote it by $A\prec B$ if $a_i\leq b_i$ for all $1\leq i\leq k$.

A family $\hf\subset \binom{[n]}{k}$ is called {\it shifted} if $A\prec B$ and $B\in \hf$ always imply $A\in \hf$. By repeated shifting one can transform an arbitrary $k$-graph into a shifted $k$-graph with the same number of edges.

We say that a family $\hf\subset \binom{[n]}{k}$ is {\it saturated $t$-intersecting} if $\hf$ is $t$-intersecting  and any addition of an extra $k$-set would destroy the $t$-intersecting property.

\begin{lem}[\cite{F78}]\label{lem-2.2}
Suppose that $\hf\subset \binom{[n]}{k}$ is shifted,  saturated and $t$-intersecting. Let $\hb=\hb_t(\hf)$.  Then $\hb \subset 2^{[2k-t]}$.
\end{lem}

\begin{proof}
Suppose the contrary and let $B\in \hb$ with $B\not\subset [2k-t]$. Choose $j\in B$ with $j>2k-t$. By the minimality of $\hb$, there exists $F\in \hf$ with $|F\cap (B\setminus\{j\})|\leq t-1$.

Then $|F\cap B|\geq t$ implies $j\in F$ and $|F\cap (B\setminus \{j\})|=|(F\setminus \{j\})\cap (B\setminus \{j\})|=t-1$. Thus,
\[
|(F\setminus \{j\})\cup (B\setminus \{j\})| =|F\setminus \{j\}|+|B\setminus \{j\}|- |(F\setminus \{j\})\cap (B\setminus \{j\})| = 2(k-1)-(t-1)=2k-1-t.
\]
Hence we can fix $i\leq 2k-t$ with $i\notin F\cup B$. Consequently by shiftedness $F'=(F\setminus \{j\})\cup\{i\}\in \hf$ and $|F'\cap B|=t-1$, a contradiction.
\end{proof}

%

\begin{lem}[\cite{F87}]\label{lem-2.4}
Let $\hf\subset \binom{[n]}{k}$ be a shifted $t$-intersecting family. Then $\hf(i)$ is $t$-intersecting for all $i>2k-t$.
\end{lem}

For a shifted $t$-intersecting family, we can prove the following result.

\begin{thm}\label{thm-2.5}
Let  $n> (t+1)(k-t)$. If $\hf\subset \binom{[n]}{k}$ is a shifted $t$-intersecting family, then
\[
d_{2k-t+1}(\hf) \leq \binom{n-t-1}{k-t-1}.
\]
\end{thm}

\begin{proof}
Let $\hf\subset \binom{[n]}{k}$ be shifted,  saturated and $t$-intersecting and let $\hb=\hb_t(\hf)$ be its basis. Then by Lemma \ref{lem-2.2} we infer that
$\hb \subset 2^{[2k-t]}$.  By saturatedness for every $i\in [2k-t+1,n]$, the link $\hf(i)$ is isomorphic, that is,
\[
\hf(i) =\left\{E\subset \binom{[n]\setminus \{i\}}{k-1}\colon \mbox{ there exists $B\in \hb$ such that } B\subset E\right\}.
\]
By Lemma \ref{lem-2.4}, $\hf(i)$ is $t$-intersecting for  $i> 2k-t$. Then by the Exact Erd\H{o}s-Ko-Rado Theorem,  for $n-1\geq (t+1)(k-t)$ we have
\[
d_i(\hf)=|\hf(i)|\leq \binom{(n-1)-t}{(k-1)-t} = \binom{n-t-1}{k-t-1}.
\]
\end{proof}

\section{Sharpening the Huang--Zhao Theorem}

By the Huang--Zhao Theorem we know that $d_n(\hf)\leq \binom{n-2}{k-2}$ for all intersecting families $\hf\subset \binom{[n]}{k}$, $n>2k$. As we have shown in Theorem \ref{thm-2.5}, in case of shifted families $d_{2k}(\hf)\leq \binom{n-2}{k-2}$, that is, already the $2k$th degree is quit small. In this section we prove  that for $n\geq 6k-9$,  $d_{2k+1}(\hf)\leq \binom{n-2}{k-2}$ holds  and for $n\geq \lceil\frac{8k}{3}\rceil$,  $d_{\lceil\frac{8k}{3}\rceil}(\hf)\leq \binom{n-2}{k-2}$ holds for general (non-shifted) intersecting families.

We need the following lemma.

\begin{lem}[\cite{FW2024}]\label{lem-key2}
Let $\hf\subset \binom{[n]}{k}$ be an intersecting family  and let $|\hf(u)|=d_1(\hf)$. If $n\geq 2k\geq 6$ and $|\hf(\bar{u},v)|\geq  5\binom{n-4}{k-3}$ for some $v\in [n]\setminus \{u\}$, then there exists $w\in [n]\setminus \{u,v\}$ such that
\begin{align*}
 |\hf(\emptyset,\{u,v,w\})|\leq \binom{n-7}{k-4}\mbox{ and }|\hf(\{x\},\{u,v,w\})|\leq \binom{n-4}{k-3},\ x=u,v,w.
\end{align*}
\end{lem}

\begin{cor}\label{cor-4.2}
Let $\hf\subset \binom{[n]}{k}$ be an intersecting family  and let $|\hf(u)|=d_1(\hf)$. If $n\geq 2k\geq 6$ and $|\hf(\bar{u},v)|\geq  5\binom{n-4}{k-3}$ for some $v\in [n]\setminus \{u\}$, then
\[
d_4(\hf) \leq 6\binom{n-3}{k-3}.
\]
\end{cor}

\begin{proof}
Indeed, by Lemma \ref{lem-key2} there exists $w\in [n]\setminus \{u,v\}$ such that
\begin{align*}
|\hf(\emptyset,\{u,v,w\})|\leq \binom{n-7}{k-4}\mbox{ and }|\hf(\{x\},\{u,v,w\})|\leq \binom{n-4}{k-3},\ x=u,v,w.
\end{align*}
It follows that for any $x\in [n]\setminus \{u,v,w\}$,
\[
|\hf(x)| \leq 3\binom{n-3}{k-3} +\binom{n-7}{k-4}+3\binom{n-4}{k-3}\leq 6\binom{n-3}{k-3}.
\]
\end{proof}

\begin{prop}\label{prop-4.1}
Let $\hf\subset \binom{[n]}{k}$ be an intersecting family  and let $|\hf(u)|=d_1(\hf)$. If $|\hf(v)|>\binom{n-2}{k-2}$ and $|\hf(\bar{u},v)|\leq \binom{n-4}{k-2}$ for some $v\in [n]\setminus \{u\}$, then
\[
|\hf(\bar{u},v)|>\frac{1}{2}|\hf(\bar{u})|.
\]
\end{prop}
\begin{proof}
Note that
\begin{align}\label{ineq-4.5}
|\hf(v)| =|\hf(u,v)|+|\hf(\bar{u},v)|>\binom{n-2}{k-2}.
\end{align}
By  $|\hf(\bar{u},v)|\leq \binom{n-4}{k-2}$,
\begin{align}\label{ineq-4.6}
|\hf(u,v)|>\binom{n-2}{k-2}-|\hf(\bar{u},v)| \geq \binom{n-2}{k-2}-\binom{n-4}{k-2}.
\end{align}
Applying Corollary \ref{cor-1} to the pair of cross-intersecting families $\hf(u,v)$, $\hf(\bar{u},\overline{v})$ yields
\[
|\hf(u,v)|+|\hf(\bar{u},\overline{v})|\leq \binom{n-2}{k-2}.
\]
Comparing with \eqref{ineq-4.5}, we get $|\hf(\bar{u},v)|>|\hf(\bar{u},\overline{v})|$.  Since $|\hf(\bar{u})|=|\hf(\bar{u},v)|+|\hf(\bar{u},\overline{v})|$, $|\hf(\bar{u},v)|>\frac{1}{2}|\hf(\bar{u})|$ follows.
\end{proof}

\begin{proof}[Proof of Theorem \ref{thm-main-4}]
 Suppose that $|\hf(1)|=d_1(\hf)$. By Corollary \ref{cor-4.2}, if $|\hf(\bar{1},x)|\geq  5\binom{n-4}{k-3}$ for some $x\in [2,n]$, then by $n\geq 6k-10$,
 \[
 d_{2k+1}(\hf)\leq d_4(\hf) \leq 6\binom{n-3}{k-3}\leq  \binom{n-2}{k-2}.
 \]
Thus we may assume $|\hf(\bar{1},x)|<  5\binom{n-4}{k-3}$ for all $x\in [2,n]$.

Note that $n\geq 6k-9$ implies that
\begin{align}\label{ineq-4.1}
|\hf(\bar{1},x)|<  5\binom{n-4}{k-3} \leq \frac{5(k-2)}{n-k-1}\binom{n-4}{k-2} \leq \binom{n-4}{k-2}.
\end{align}
Arguing indirectly we may assume that there exist $2k$ distinct elements $x_1,\ldots,x_{2k}\in [2,n]$ satisfying $|\hf(x_i)|>\binom{n-2}{k-2}$. Applying Proposition \ref{prop-4.1} with $u=1$ and $v=x_i$, we obtain that
 \[
 |\hf(\bar{1},x_i)|>\frac{1}{2}|\hf(\bar{1})|,\ 1\leq i\leq 2k.
 \]
 Summing these $2k$ inequalities,
\[
k|\hf(\bar{1})| =\sum_{x\in [2,n]} |\hf(\bar{1},x)|\geq \sum_{1\leq i\leq 2k} |\hf(\bar{1},x_i)|>k|\hf(\bar{1})|,
\]
a contradiction. Thus $d_{2k+1} \leq \binom{n-2}{k-2}$.
\end{proof}

\begin{lem}\label{lem-4.4}
Let $\hf\subset \binom{[n]}{k}$ be an intersecting family with $n>2k$. Then either $d_{2k+1}(\hf)\leq \binom{n-2}{k-2}$ or $|\hf|\leq  \binom{n-2}{k-2}+2\binom{n-3}{k-2}$ holds.
\end{lem}

\begin{proof}
Suppose that $|\hf(1)|=d_1(\hf)$. Then the {\it diversity}, $\gamma(\hf)$ is simply $|\hf(\bar{1})| =|\hf|-d_1(\hf)$.

Let us show that  $\gamma(\hf)\leq \binom{n-4}{k-2}$ implies $d_{2k+1}(\hf)\leq \binom{n-2}{k-2}$.
Arguing indirectly assume that there exist $2k$ distinct elements $x_1,\ldots,x_{2k}\in [2,n]$ satisfying $|\hf(x_i)|>\binom{n-2}{k-2}$. Note that
\[
|\hf(\bar{1},x_i)|\leq |\hf(\bar{1})|= \gamma(\hf)\leq\binom{n-4}{k-2}.
\]
 Applying Proposition \ref{prop-4.1} with $u=1$ and $v=x_i$, we obtain that
 \[
 |\hf(\bar{1},x_i)|>\frac{1}{2}|\hf(\bar{1})|,\ 1\leq i\leq 2k.
 \]
 Summing these $2k$ inequalities,
\[
k|\hf(\bar{1})| =\sum_{x\in [2,n]} |\hf(\bar{1},x)|\geq \sum_{1\leq i\leq 2k} |\hf(\bar{1},x_i)|>k|\hf(\bar{1})|,
\]
a contradiction. Thus we may assume that $\gamma(\hf)> \binom{n-4}{k-2}$.

For $n> 2k$, $|\hf(\bar{1})|=\gamma(\hf)> \binom{n-4}{k-2}\geq  \binom{n-4}{k-3}$. By applying Corollary \ref{cor-hm} with $\ha=\hf(\bar{1})$ and $\hb=\hf(1)$,
\[
|\hf(1)|\leq \binom{n-1}{k-1}-\binom{n-4}{k-1}.
\]
By Theorem \ref{thm-f87}, we conclude that
\begin{align*}
|\hf|\leq |\hh_3|=\binom{n-2}{k-2}+2\binom{n-3}{k-2}.
\end{align*}
Note that this can be deduced directly from Theorem 2 in \cite{KZ2018} as well.
\end{proof}

\begin{proof}[Proof of Theorem \ref{thm-main-6}]
By Lemma \ref{lem-4.4} we may assume that $|\hf|\leq  \binom{n-2}{k-2}+2\binom{n-3}{k-2}$. If $n\geq 6k-9$, then by Theorem \ref{thm-main-4},
\[
d_{\lceil\frac{8k}{3}\rceil}(\hf) \leq d_{2k+1}(\hf) \leq \binom{n-2}{k-2}.
\]
Thus we may assume $n\leq 6k-10$.

Note that $n\leq 6k-10$ is equivalent to  $\frac{k-2}{n-2} \geq \frac{1}{6}$. Then
\begin{align}\label{ineq-4.7}
|\hf|\leq \binom{n-2}{k-2}+2\binom{n-3}{k-2}=\left(3-\frac{2(k-2)}{n-2}\right)\binom{n-2}{k-2}\leq \frac{8}{3}\binom{n-2}{k-2}.
\end{align}
If $d_{\lceil\frac{8k}{3}\rceil}>\binom{n-2}{k-2}$, then
\[
\left\lceil\frac{8k}{3}\right\rceil\binom{n-2}{k-2}< \sum_{1\leq i\leq \lceil\frac{8k}{3}\rceil}d_{i}\leq  \sum_{x\in [n]}|\hf(x)| = k|\hf|<\frac{8}{3}k\binom{n-2}{k-2},
\]
a contradiction. Thus $d_{\lceil\frac{8k}{3}\rceil}\leq \binom{n-2}{k-2}$ holds.
\end{proof}

Let us prove a statement that assumes $\tau(\hf)= 2$.

\begin{prop}
Let $\hf\subset \binom{n}{k}$ be an intersecting family with  $\tau(\hf)=2$. Then for $n>2k$,
\[
d_{2k+1}(\hf)<\binom{n-2}{k-2}.
\]
\end{prop}

\begin{proof}
Without loss of generality we assume that $\hf$ is saturated. By Lemma \ref{lem-4.4} we may assume that $|\hf|\leq  \binom{n-2}{k-2}+2\binom{n-3}{k-2}$.  Argue by contradiction, let $\{x_1,x_2\}$ be a transversal and suppose that $z_1,\ldots,z_{2k-2}\in [n]\setminus \{x_1,x_2\}$ have degree at least $\binom{n-2}{k-2}$. By saturatedness, we have $|\hf(\{x_1,x_2\})|=\binom{n-2}{k-2}$. Let
\[
\tilde{\hf}=\hf\setminus \{F\in \hf\colon \{x_1,x_2\}\subset F\}.
\]
Then $|\tilde{\hf}|< 2\binom{n-3}{k-2}$ and  $|\tilde{\hf}(z_i)|\geq \binom{n-2}{k-2}-\binom{n-3}{k-3}=\binom{n-3}{k-2}$, $i=1,2,\ldots,2k-2$. Since $F\cap \{x_1,x_2\}\neq \emptyset$ for each $F\in \tilde{\hf}$,
\[
2(k-1)\binom{n-3}{k-2} \leq \sum_{1\leq i\leq  2k-2} |\tilde{\hf}(z_i)| \leq \sum_{x\in [n]\setminus \{x_1,x_2\}} |\tilde{\hf}(x)|< (k-1)|\tilde{\hf}|<2(k-1)\binom{n-3}{k-2},
\]
a contradiction. Thus at  most $2k-3$ vertices in $[n]\setminus \{x_1,x_2\}$ have degree at least $\binom{n-2}{k-2}$ and the proposition follows.
\end{proof}

\section{Stronger bounds for large $n$}

We continue assuming that $\hf\subset \binom{[n]}{k}$ is a saturated $t$-intersecting family and $n> 2k-t$.
Recall that  $\hht=\hht_t(\hf)$ is the family of $t$-transversals of size not exceeding $k$ and the base  $\hb=\hb_t(\hf)$ is the family of members of $\hht$ that are minimal for containment.

Let us first prove a  statement for $\tau_t(\hf)\geq t+2$.

\begin{prop}\label{prop-new5.1}
Let $\hf\subset  \binom{[n]}{k}$ be saturated and  $t$-intersecting. If $n\geq \binom{t+2}{2}k^2 $ and $\tau_t(\hf)\geq t+2$ then
\begin{align}\label{ineq-new5.1}
d_{1+\tau_t(\hf)} <\binom{n-t-1}{k-t-1}
\end{align}
\end{prop}

\begin{proof}
Set $r=\tau_t(\hf)$ and fix $T_1\in \hht$ satisfying $|T_1|=r\geq t+2$. To prove \eqref{ineq-new5.1} it is sufficient to show
\begin{align}\label{ineq-new5.2}
|\hf(y)|<\binom{n-t-1}{k-t-1} \mbox{ for all }y\in [n]\setminus T_1.
\end{align}

To prove \eqref{ineq-new5.2} for a fixed $y\in [n]\setminus T_1$ let us construct $rk^{r-t}$ sequences $(z_1,z_2,\ldots,z_{r-1})$ in the following way.

Choose a $t$-subset
$\{z_1,z_2,\ldots,z_t\}\subset T_1$, define a sequence $(z_1,z_2,\ldots,z_t)$. If $z_1,\ldots,z_\ell$ are chosen and $\ell<r-1$ then choose $F_{\ell+1}\in \hf$ satisfying $|\{y,z_1,z_2,\ldots,z_\ell\}\cap F_{\ell+1}|<t$. This is possible by $\tau_t(\hf)=r$. Now  extend the sequence to $(z_1,z_2,\ldots,z_{\ell+1})$ in all possible ways with $z_{\ell+1}\in F_{\ell+1}\setminus \{y,z_1,z_2,\ldots,z_\ell\}$.

It should be clear from the construction that every member $G\in \hf(y)$ contains at least one of the sets $\{z_1,z_2,\ldots,z_\ell\}$ at each stage, in particular for $\ell=r-1$. Consequently,
\[
|\hf(y)| \leq \binom{r}{t}\cdot k^{r-1-t} \binom{n-r}{k-r}.
\]

To conclude the proof of \eqref{ineq-new5.2} we show that for $r\geq t+2$ and $n\geq \binom{t+2}{2}k^2$,
\begin{align}\label{ineq-new5.3}
\binom{r}{t}\cdot k^{r-1-t} \binom{n-r}{k-r}<\binom{n-t-1}{k-t-1}.
\end{align}
For $r=t+2$,
$\binom{t+2}{2}k\binom{n-t-2}{k-t-2}=\binom{t+2}{2}\frac{k(k-t-1)}{n-t-1}\binom{n-t-1}{k-t-1}$ and the desired inequality follows from $n\geq \binom{t+2}{2}k^2 >\binom{t+2}{2}k(k-t-1)+t+1$.  For $r\geq t+3$ just note that the LHS of \eqref{ineq-new5.3} is a monotone decreasing function of $t$:
\[
\frac{\binom{r+1}{t}k^{r-t}\binom{n-r-1}{k-r-1}}{\binom{r}{t}k^{r-1-t}\binom{n-r}{k-r}} = \frac{r+1}{r+1-t} \frac{k(k-r)}{n-r} <\frac{t+4}{4} \frac{k^2}{n}<1.
\]
\end{proof}

A collection of sets $F_0,\ldots,F_k$ is called a {\it sunflower of size $k+1$ with  center $C$} if $F_i\cap F_j=C$ for all distinct $i,j\in \{0,1,\ldots,k\}$.

\begin{proof}[Proof of Theorem \ref{thm-main-5}]
Let $\hf\subset \binom{[n]}{k}$ be saturated and $t$-intersecting. By Proposition \ref{prop-new5.1} we may assume $\tau_t(\hf)\leq t+1$.
Let $\hht=\hht_t(\hf)$, $\hb=\hb_t(\hf)$ and $\hb_0=\{B\in \hb\colon |B|=t+1\}$. If there exists $B\in \hb$ with $|B|=t$, then  $B\subset F$ for all $F\in \hf$. By saturatedness  $d_{k+2}(\hf)=\binom{n-t-1}{k-t-1}$ and we are done. Thus we may assume $\tau_t(\hf)= t+1$ and $\hb_0\neq \emptyset$.

\begin{claim}\label{claim-5.4}
$|\cup \hb_0|\leq k+1$.
\end{claim}

\begin{proof}
Suppose for contradiction that $|\cup \hb_0|\geq k+2$. Choose $B_{t+1},B_{t+2}\in \hb_0$. Then by symmetry we may assume $B_{t+i}=[t]\cup \{t+i\}$, $i=1,2$. For $t+2<j\leq k+2$ fix $B_j\in \hb_0$ with $j\in \hb_j$. Now $|B_j|=t+1$ and $|B_{t+i}\cap B_j|\geq  t$ imply $B_j=[t]\cup \{j\}$. Hence $B_{t+1},B_{t+2},\ldots,B_{k+2}$ form a sunflower with center $[t]$. Now $[t]\subset F$ follows for all $F\in \hf$, contradicting our assumption $\tau_t(\hf)= t+1$.
\end{proof}

By Claim \ref{claim-5.4} and symmetry we may assume $\cup \hb_0 \subset [k+1]$. To prove the theorem it is sufficient to show
\begin{align}
|\hf(x)| <\binom{n-t-1}{k-t-1} \mbox{ for }k+2<x\leq u.
\end{align}
We distinguish two cases.

{\bf Case 1. }  $\hb_0$ is not a sunflower and $|\hb_0|\geq 3$.

Without loss of generality, assume  $[t+1],[t]\cup \{t+2\}\in \hb_0$. Since $\hb_0$ is not a sunflower, there is some $B_0\in \hb_0$ such that $|B_0\cap [t]|<t$. Since $|B_0\cap [t+1]|\geq t$, it follows that $|B_0\cap [t]|=t-1$.
Since $|B_0\cap ([t]\cup \{t+2\})|\geq t$, we have $\{t+1,t+2\}\subset B_0$ whence $B_0\subset [t+2]$. Now  the $t$-intersecting property implies $|F\cap [t+2]|\geq t+1$ for all  $F\in \hf$. Indeed, $F\cap [t+1]$, $F\cap ([t]\cup \{t+2\})$ and $F\cap B_0$, all three have size at least $t$. This implies $|F\cap [t+2]|\geq t+1$ for all $F\in \hf$. Thus for any $x\in [t+3,n]$,
\[
|\hf(x)| \leq (t+2)\binom{n-t-2}{k-t-2} <\binom{n-t-1}{k-t-1} \mbox{ for }n\geq (t+2)k.
\]

{\bf Case 2. }$\hb_0$ is a sunflower.

 Without loss of generality, assume that $\hb_0=\{[t]\cup \{q\}\colon t+1\leq q\leq t+s\}$ for some $s$, $1\leq s\leq k-t+1$. Since $[t]\notin \hb_0$, there exists $F_0\in \hf$ such that $|F_0\cap[t]|\leq t-1$. Since $|F_0\cap ([t]\cup \{q\})|\geq t$ for all $t+1\leq q\leq t+s$, it follows that $|F_0\cap[t]|= t-1$ and $\{t+1,t+2,\ldots, t+s\}\subset F_0$. Then $|[t]\cup F_0|=k+1$. Let $x\in [n]\setminus ([t]\cup F_0)$. For any $F\in \hf$ with $x\in F$, we infer that either $[t]\subset F$ and $F\cap (F_0\setminus [t])\neq \emptyset$ or $|F\cap [t]|=t-1$ and $\{t+1,t+2,\ldots, t+s\} \subset F$. Then
\begin{align}
|\hf(x)| \leq (k-t+1) \binom{n-t-2}{k-t-2} +t \binom{n-t-s-1}{k-t-s}.
\end{align}
For $s\geq 2$ and $n\geq k^2$,
\[
|\hf(x)| \leq (k-t+1) \binom{n-t-2}{k-t-2} +t \binom{n-t-3}{k-t-2}< \binom{n-t-1}{k-t-1}.
\]

For $s=1$, assume $\hb_0=\{[t+1]\}$. Let $x\in [t+2,n]$. For any $T\in \binom{[t+1]}{t}$, since $T\cup \{x\}\notin \hb_0$, there exists $F_T$ such that $|F_T\cap (T\cap \{x\})|\leq t-1$. It follows that for $n\geq (t+1)(k-t)^2$,
\[
|\hf(x)| \leq \binom{t+1}{t} (k-t+1) \binom{n-t-2}{k-t-2} < \binom{n-t-1}{k-t-1}.
\]
\end{proof}

\section{Proofs of Theorems \ref{thm-main-2} and \ref{thm-main-3}}

\begin{lem}\label{lem-3.4}
Let $\ha,\hb\subset \binom{[n]}{k}$ be cross-intersecting, $|\ha|\leq |\hb|$ and $n\geq (r+1)k$. Then
\begin{align}
r|\ha| +|\hb| \leq \binom{n}{k}.
\end{align}
\end{lem}

\begin{proof}
By the Kruskal--Katona Theorem we may assume that $\ha,\hb$ are initial segments in the lexicographic order. Hence $\ha\subset \hb$ and $1\in A$ for all  $A\in \ha$. Define a bipartite graph on $X=\ha(1)$, $Y=\binom{[2,n]}{k}$ by putting an edge iff the two sets are disjoint. The degree of $A\setminus \{1\}$ for $A\in \ha$ is $\binom{n-k}{k}$. On the other hand for $B\in Y$ its degree is at most $\binom{n-k-1}{k-1} \leq \frac{1}{r}\binom{n-k}{k}$. It follows that
\[
|\ha| \binom{n-k}{k} \leq \frac{1}{r}\binom{n-k}{k} |N(\ha)|,
\]
where $N(\ha)$ is the set of  neighbors of $\ha$.
Thus the neighborhood of $\ha$ in $Y$ has size at least $r|\ha|$ and these $k$-sets are not in $\hb$.
\end{proof}

\begin{lem}[\cite{FW2024}]\label{lem-key}
Let $\ha\subset\binom{[n]}{a}$, $\hb\subset \binom{[n]}{b}$ be cross-intersecting families with $n\geq a+b$. If $|\ha|\geq \binom{n-1}{a-1}+\binom{n-2}{a-1}+\ldots+\binom{n-d}{a-1}$, $d<b$, then
\[
|\ha|+\frac{\binom{n-d}{a}}{\binom{n-d}{b-d}}|\hb| \leq \binom{n}{a}.
\]
\end{lem}

\begin{proof}[Proof of Theorem \ref{thm-main-2}]
Let $\hf\subset \binom{[n]}{k}$ be a saturated intersecting family  and let  $|\hf(1)|$, $|\hf(2)|$, $|\hf(3)|$, $|\hf(4)|$ be the largest four degrees. To prove the theorem it suffices to show that there exists $i\in [4]$ such that
\[
 f(i):= |\hf(i)|\leq \binom{n-2}{k-2}+\binom{n-4}{k-3}.
\]

Set $f(Q) = |\hf(Q,[4])|$ for $Q\subset [4]$. Then
\begin{align}\label{ineq-3.5}
f(i) = |\hf(i)|= \sum_{i\in Q\subset [4]} f(Q).
\end{align}

For the proof it is useful to break up \eqref{ineq-3.5} into two parts:
\begin{align*}
f_{small}(i) = f(\{i\}) + \sum_{Q\in \binom{[4]}{2},i\in Q} f(Q) \mbox { and }
f_{big}(i) =\sum_{i\in Q\subset [4],\ |Q|\geq 3} f(Q).
\end{align*}

\begin{claim}\label{claim-6.1}
If $f_{small}(i)\leq \binom{n-4}{k-2}$ then $f(i)\leq \binom{n-2}{k-2}+\binom{n-4}{k-3}$.
\end{claim}
\begin{proof}
Evidently, $f_{big}(i) \leq 3\binom{n-4}{k-3}+\binom{n-4}{k-4}$.
Hence,
\[
f(i)= f_{small}(i)+f_{big}(i)\leq \binom{n-2}{k-2}+\binom{n-4}{k-3}.
\]
\end{proof}
By Claim \ref{claim-6.1}, it suffices to show that for some $i\in [4]$,
\begin{align}\label{ineq-small}
f_{small}(i)\leq \binom{n-4}{k-2}.
\end{align}

Consider $P,P'\in \binom{[4]}{2}$ where $P=[4]\setminus P'$. Then
$\hf(P,[4])$ and $\hf(P',[4])$ are cross-intersecting. Let us make a graph $\hg$ on the vertex set $[4]$ by drawing an edge $P$ iff $f(P)\geq \binom{n-5}{k-3}+\binom{n-6}{k-3}$. By the Daykin Theorem there are no two disjoint edges. Hence either $\hg$ is a triangle or a star. We distinguish four cases according to the structure of $\hg$.

\vspace{3pt}
{\bf Case 1.} $\hg\neq \emptyset$ and there exists an isolated vertex in $\hg$.
\vspace{3pt}

Without loss of generality, let $\{1,2\}\in \hg$ and let 4 be isolated in $\hg$.
Since $\hf(\{4\},[4])$, $\hf(\{1,2\},[4])$ are cross-intersecting and $f(\{1,2\})\geq \binom{n-5}{k-3}+\binom{n-6}{k-3}$, by the Kruskal--Katona Theorem we infer that
\[
f(\{4\}) \leq \binom{n-6}{k-3}.
\]
Since $\hf(\{1,2\},[4])$, $\hf(\{3,4\},[4])$ are cross-intersecting and $f(\{1,2\})\geq \binom{n-5}{k-3}+\binom{n-6}{k-3}$, by the Kruskal--Katona Theorem
\[
f(\{3, 4\}) \leq \binom{n-6}{k-4}.
\]
Since 4 is isolated, by the definition of $\hg$ we have $f(\{j,4\})<\binom{n-5}{k-3}+\binom{n-6}{k-3}$ for $j=1,2$. Then for $(n-4)\geq 5(k-2)$,
\begin{align*}
f_{small}(4) = f(\{4\})+\sum_{1\leq j\leq 3} f(\{j,4\}) &< \binom{n-6}{k-3}+\binom{n-6}{k-4}+2\left(\binom{n-5}{k-3}+\binom{n-6}{k-3}\right)\\[3pt]
&<5\binom{n-5}{k-3}\\[3pt]
&\leq \binom{n-4}{k-2}.
\end{align*}

\vspace{3pt}
{\bf Case 2.} $\hg=\{\{1,2\},\{1,3\},\{1,4\}\}$.
\vspace{3pt}

By symmetry assume that $f(\{1,2\})\geq f(\{1,3\})\geq f(\{1,4\})$. Note that $n\geq 4k-3$ implies
\[
\frac{\binom{n-6}{k-2}}{\binom{n-6}{k-4}}\geq \frac{\binom{n-6}{k-2}}{\binom{n-6}{k-3}}=\frac{n-k-3}{k-2}\geq 3.
 \]
 Since $\hf(\{1,2\},[4])$, $\hf(\{4\},[4])$  are cross-intersecting and $|\hf(\{1,2\},[4])|\geq \binom{n-5}{k-3}+\binom{n-6}{k-3}$, by Lemma \ref{lem-key}   we get
\begin{align}\label{ineq-3.10}
 f(\{1,4\})+3f(\{4\})\leq f(\{1,2\})+\frac{\binom{n-6}{k-2}}{\binom{n-6}{k-3}}f(\{4\})\leq \binom{n-4}{k-2}.
\end{align}
Since  $\hf(\{1,2\},[4])$, $\hf(\{3,4\},[4])$ are cross-intersecting, by Lemma \ref{lem-key} we get
\begin{align}\label{ineq-3.11}
 f(\{1,4\})+3f(\{3,4\})\leq f(\{1,2\})+\frac{\binom{n-6}{k-2}}{\binom{n-6}{k-4}}f(\{3,4\}) \leq \binom{n-4}{k-2}.
\end{align}
Similarly,
\begin{align}\label{ineq-3.12}
 f(\{1,4\})+3f(\{2,4\})\leq f(\{1,3\})+3f(\{2,4\})\leq \binom{n-4}{k-2}.
\end{align}
Adding \eqref{ineq-3.10}, \eqref{ineq-3.11}, \eqref{ineq-3.12} and dividing by $3$, we obtain that
\[
f_{small}(4) = f(\{1,4\})+ f(\{4\})+f(\{2,4\})+f(\{3,4\})\leq \binom{n-4}{k-2}.
\]

\vspace{3pt}
{\bf Case 3.} $\hg=\emptyset$ and  $f(\{i\})\geq \binom{n-6}{k-3}+\binom{n-7}{k-3}$ for $i=1,2,3,4$.
\vspace{3pt}

Note that $\hf(\{i\},[4])$, $\hf(\{j\},[4])$ are cross-intersecting for $1\leq i<j\leq 4$. By the Daykin Theorem, one of $\hf(\{i\},[4])$ and $\hf(\{j\},[4])$ has size at most $\binom{n-5}{k-2}$. Without loss of generality assume $|\hf(\{4\},[4])|\leq \binom{n-5}{k-2}$.  Since $|\hf(\{i\},[4])|=f(\{i\})\geq \binom{n-6}{k-3}+\binom{n-7}{k-3}$ for all $i=1,2,3,4$, by Corollary  \ref{cor-hm},
\begin{align}\label{ineq-3.8}
|\hf(\{j,4\})|\leq  \binom{n-5}{k-3}+\binom{n-7}{k-4} \mbox{ for } j=1,2,3
\end{align}
and
\begin{align}\label{ineq-3.9}
|\hf([4]\setminus \{j\})|\leq  \binom{n-5}{k-4}+\binom{n-7}{k-5}\mbox{ for } j=1,2,3.
\end{align}
Adding \eqref{ineq-3.8}, \eqref{ineq-3.9} and $|\hf([4],[4])|\leq \binom{n-4}{k-4}$,
\begin{align*}
f(4) &=f(\{4\})+\sum_{1\leq j\leq 3}f(\{j,4\})+\sum_{1\leq j\leq 3}f([4]\setminus \{j\}) +f([4])\\[3pt]
&\leq \binom{n-5}{k-2}+3\left(\binom{n-5}{k-3}+\binom{n-7}{k-4}\right)+3\left(\binom{n-5}{k-4}+\binom{n-7}{k-5}\right) +\binom{n-4}{k-4}.
\end{align*}
Using $\binom{n-2}{k-2} = \binom{n-5}{k-2}+3\binom{n-5}{k-3}+3\binom{n-5}{k-4}+\binom{n-5}{k-5}$, $\binom{n-7}{k-4}+\binom{n-7}{k-5}=\binom{n-6}{k-4}$ and $n\geq  4k$,  we obtain that
\begin{align*}
f(4)&\leq   \binom{n-2}{k-2}+ 3\binom{n-6}{k-4}+\binom{n-4}{k-4}-\binom{n-5}{k-5}\\[3pt]
&\leq  \binom{n-2}{k-2}+ 4\binom{n-5}{k-4}\\[3pt]
&\leq  \binom{n-2}{k-2}+ \binom{n-4}{k-3}.
\end{align*}

\vspace{3pt}
{\bf Case 4. } $\hg=\emptyset$ and $f(\{i\})< \binom{n-6}{k-3}+\binom{n-7}{k-3}$ for some $i\in [4]$.
\vspace{3pt}

Without loss of generality assume that
\begin{align}\label{ineq-6.8}
f(\{4\})<\binom{n-6}{k-3}+\binom{n-7}{k-3}.
\end{align}
Note that $\hg=\emptyset$ implies
\begin{align}\label{ineq-6.7}
f(\{j,4\})\leq \binom{n-5}{k-3}+\binom{n-6}{k-3}.
\end{align}
 If $f(\{i\})\geq \binom{n-6}{k-3}+\binom{n-7}{k-3}$ for some $i\in [3]$, by Corollary \ref{cor-hm} we infer that
\begin{align}\label{ineq-6.9}
f(\{j,4\}) \leq \binom{n-5}{k-3}+\binom{n-7}{k-4}\mbox{ for }j\neq i.
\end{align}
Let
\[
I=\left\{i\in [3]\colon f(\{i\})\geq \binom{n-6}{k-3}+\binom{n-7}{k-3} \right\}.
\]

{\bf Subcase 4.1. } $|I|\geq 2$.

By \eqref{ineq-6.8}, \eqref{ineq-6.9} and $n\geq 5k$,
\begin{align*}
f_{small}(4) =\sum_{1\leq j\leq 3} f(\{j,4\}) + f(\{4\}) &\leq 3\left(\binom{n-5}{k-3}+\binom{n-7}{k-4}\right)+\binom{n-6}{k-3}+\binom{n-7}{k-3}\\[3pt]
&< 5\binom{n-5}{k-3} \leq \binom{n-4}{k-2}.
\end{align*}

{\bf Subcase 4.2. }$|I|=1$.

By \eqref{ineq-6.8}, \eqref{ineq-6.7}, \eqref{ineq-6.9} and  $n\geq 6k$,
\begin{align*}
f_{small}(4) &=\sum_{1\leq j\leq 3} f(\{j,4\}) + f(\{4\})\\[3pt]
 &\leq 2\left(\binom{n-5}{k-3}+\binom{n-7}{k-4}\right)+\binom{n-5}{k-3}+\binom{n-6}{k-3}+\binom{n-6}{k-3}+\binom{n-7}{k-3}\\[3pt]
&< 6\binom{n-5}{k-3} \leq \binom{n-4}{k-2}.
\end{align*}

{\bf Subcase 4.3. }$I=\emptyset$.

By the definition of $I$,
\begin{align}\label{ineq-6.10}
f(i) < \binom{n-6}{k-3}+\binom{n-7}{k-3}, \ i=1,2,3,4.
\end{align}
Let $\hg'$ be a graph on the vertex set $[4]$ with  $P\in \binom{[4]}{2}$ being an edge iff $f(P)>\binom{n-5}{k-3}$. By the Daykin Theorem there are no two disjoint edges.

If there is an isolated vertex of $\hg'$, without loss of generality assume 4 is isolated,   then $\{j,4\}\notin \hg'$ for $j=1,2,3$, implying $f(\{j,4\})\leq \binom{n-5}{k-3}$. By \eqref{ineq-6.10} and $n\geq 5k$,
\begin{align*}
f_{small}(4) &=\sum_{1\leq j\leq 3} f(\{j,4\}) + f(\{4\})\\[3pt]
 &\leq 3\binom{n-5}{k-3}+\binom{n-6}{k-3}+\binom{n-7}{k-3}\\[3pt]
&< 5\binom{n-5}{k-3} \leq \binom{n-4}{k-2}.
\end{align*}
Thus we may assume $\hg'=\{\{1,2\},\{1,3\},\{1,4\}\}$. Then $\{1,4\},\{2,4\}\notin \hg'$ implies $f(\{1,4\})\leq \binom{n-5}{k-3}$ and $f(\{2,4\})\leq \binom{n-5}{k-3}$. Thus by \eqref{ineq-6.7}, \eqref{ineq-6.10} and $n\geq 6k$
\begin{align*}
f_{small}(4) &=\sum_{1\leq j\leq 3} f(\{j,4\}) + f(\{4\})\\[3pt]
 &\leq 2\binom{n-5}{k-3}+\binom{n-5}{k-3}+\binom{n-6}{k-3}+\binom{n-6}{k-3}+\binom{n-7}{k-3}\\[3pt]
&< 6\binom{n-5}{k-3} < \binom{n-4}{k-2}.
\end{align*}
\end{proof}

\begin{proof}[Proof of Theorem \ref{thm-main-3}]
Let $\hf\subset \binom{[n]}{k}$ be a saturated intersecting family. Let  $|\hf(1)|$, $|\hf(2)|$, $\ldots,|\hf(\ell+1)|$ be the largest $\ell+1$ degrees. We are going to show that for some $i\in [\ell+1]$,
\[
f(i) := |\hf(i)|\leq \binom{n-2}{k-2}+\binom{n-\ell-1}{k-\ell}.
\]

 Set $f(Q) = |\hf(Q,[\ell+1])|$ for $Q\subset [\ell+1]$. Then
\begin{align}\label{ineq-thm3.5-1}
f(i)= \sum_{i\in Q\subset [\ell+1]} f(Q).
\end{align}
For $p\in [\ell+1]\setminus \{i\}$, we define
\begin{align*}
&f_1(i) = \sum_{i\in Q\subset [\ell+1],\ 1\leq |Q|\leq 2} f(Q),\\[3pt]
&f_2(i,p) = \sum_{i\in Q\subset [\ell+1],\ p\notin Q,\ 3\leq |Q|\leq \ell-1} f(Q),\\[3pt]
&f_3(i,p) = \sum_{i\in Q\subset [\ell+1],\ |Q|\geq \ell} f(Q)+\sum_{i\in Q\subset [\ell+1],\ p\in Q,\ 3\leq |Q|\leq \ell-1} f(Q).
\end{align*}
Clearly $f(i)=f_1(i)+f_2(i,p)+f_3(i,p)$ holds for every $p\in [\ell+1]\setminus \{i\}$.

\begin{claim}\label{claim-6.2}
If $f_1(i)\leq \frac{1}{2}\binom{n-\ell-1}{k-2}$ or $f_1(i)+f_2(i,p)\leq \binom{n-\ell-1}{k-2}$ for some $p\in [\ell+1]\setminus \{i\}$, then $f(i)\leq \binom{n-2}{k-2}+\binom{n-\ell-1}{k-\ell}$.
\end{claim}
\begin{proof}
Evidently,
\begin{align*}
f_3(i,p)&= \sum_{i\in Q\subset [\ell+1],\ |Q|\geq \ell} f(Q)+\sum_{i\in Q\subset [\ell+1],\ p\in Q,\ 3\leq |Q|\leq \ell-1} f(Q)\\[3pt]
 &\leq \ell \binom{n-\ell-1}{k-\ell}+\binom{n-\ell-1}{k-\ell-1}+\sum_{1\leq j\leq \ell-3}\binom{\ell-1}{j}\binom{n-\ell-1}{k-2-j}\\[3pt]
&=  \binom{n-\ell-1}{k-\ell}+\sum_{1\leq j\leq \ell-1}\binom{\ell-1}{j}\binom{n-\ell-1}{k-2-j}\\[3pt]
&=\binom{n-\ell-1}{k-\ell} +\binom{n-2}{k-2}-\binom{n-\ell-1}{k-2}.
\end{align*}
If $f_1(i)+f_2(i,p)\leq \binom{n-\ell-1}{k-2}$ for some $p\in [\ell+1]\setminus \{i\}$, then $f(i)\leq \binom{n-2}{k-2}+\binom{n-\ell-1}{k-\ell}$ follows.
If $f_1(i)\leq \frac{1}{2}\binom{n-\ell-1}{k-2}$, we are left to show that \begin{align}\label{ineq-6.11}
f_2(i,p) < \frac{1}{2}\binom{n-\ell-1}{k-2} \mbox{ for some }p\in [\ell+1]\setminus \{i\}.
\end{align}
Let us fix some $p\in [\ell+1]\setminus \{i\}$. Since
\begin{align*}
f_2(i,p) &= \sum_{i\in Q\subset [\ell+1],\ p\notin Q,\ 3\leq |Q|\leq \ell-1} f(Q)\leq \sum_{3\leq j\leq \ell-1} \binom{\ell-1}{j-1} \binom{n-\ell-1}{k-j}
\end{align*}
and $n\geq 2\ell k$ implies
\[
\frac{\binom{\ell-1}{j-1} \binom{n-\ell-1}{k-j}}{\binom{\ell-1}{j} \binom{n-\ell-1}{k-j-1}} = \frac{j}{\ell-j} \cdot \frac{n-k-\ell+j}{k-j}  > 2,
\]
it follows that for $n\geq 2\ell^2 k$,
\[
f_2(i,p) \leq 2\binom{\ell-1}{2} \binom{n-\ell-1}{k-3} <(\ell-1)(\ell-2) \binom{n-\ell-1}{k-3}< \frac{1}{2}\binom{n-\ell-1}{k-2}
\]
and \eqref{ineq-6.11} holds.
\end{proof}

By Claim \ref{claim-6.2}, it suffices to show that for some $i\in [\ell+1]$,
\begin{align}\label{ineq-thm3.5-small}
f_1(i)\leq \frac{1}{2}\binom{n-\ell-1}{k-2}
\end{align}
or for some $p\in [\ell+1]\setminus \{i\}$,
\begin{align}\label{ineq-thm3.5-small2}
f_1(i)+f_2(i,p)\leq \binom{n-\ell-1}{k-2}.
\end{align}

Consider $P,P'\in \binom{[\ell+1]}{2}$ with $P\cap P'=\emptyset$. Then
$\hf(P,[\ell+1])$ and $\hf(P',[\ell+1])$ are cross-intersecting. Let us make a graph $\hg$ on vertex set $[\ell+1]$, drawing an edge $P$ iff $f(P)\geq \binom{n-\ell-2}{k-3}+\binom{n-\ell-3}{k-3}$. By the Daykin Theorem there are no two disjoint edges. Hence either $\hg$ is a triangle or a star. We distinguish four cases according to the structure of $\hg$.

\vspace{3pt}
{\bf Case 1.} $\hg\neq \emptyset$ and there exists an isolated vertex in $\hg$.
\vspace{3pt}

Without loss of generality, let $\{1,2\}\in \hg$ and $\ell+1$ be isolated.
Since $\hf(\{\ell+1\},[\ell+1])$, $\hf(\{1,2\},[\ell+1])$ are cross-intersecting and $f(\{1,2\})\geq\binom{n-\ell-2}{k-3}+\binom{n-\ell-3}{k-3}$, by the Kruskal--Katona Theorem we have
\[
f(\{\ell+1\}) \leq \binom{n-\ell-3}{k-3}.
\]
Since $\ell+1$ is isolated, by the definition of $\hg$ we infer $f(\{j,\ell+1\})<\binom{n-\ell-2}{k-3}+\binom{n-\ell-3}{k-3}$ for $j=1,2,\ldots, \ell$. It follows that for $n\geq (4\ell+2)k$,
\begin{align*}
 f_1(\ell+1)&=f(\{\ell+1\})+\sum_{1\leq j\leq \ell} f(\{j,\ell+1\})\\[3pt]
  &< \binom{n-\ell-3}{k-3}+\ell\left(\binom{n-\ell-2}{k-3}+\binom{n-\ell-3}{k-3}\right)\\[3pt]
 &<(2\ell+1)\binom{n-\ell-2}{k-3}\\[3pt]
 &<  \frac{1}{2}\binom{n-\ell-1}{k-2},
\end{align*}
which proves \eqref{ineq-thm3.5-small}.

\vspace{3pt}
{\bf Case 2.} $\hg$ is a star of size $\ell$.
\vspace{3pt}

Without loss of generality, assume that $\hg=\{\{1,2\},\{1,3\},\ldots,\{1,\ell+1\}\}$ and
$f(\{1,2\})\geq f(\{1,3\})\geq \ldots \geq f(\{1,\ell+1\})$.
Note the $n\geq (2\ell+1)k$ implies
\[
\frac{\binom{n-\ell-3}{k-2}}{\binom{n-\ell-3}{k-4}}\geq \frac{\binom{n-\ell-3}{k-2}}{\binom{n-\ell-3}{k-3}} = \frac{n-\ell-k}{k-2}> 2\ell.
\]
Since $\hf(\{1,2\},[\ell+1])$, $\hf(\{\ell+1\},[\ell+1])$ are cross-intersecting and $|\hf(\{1,2\},[\ell+1])|\geq \binom{n-\ell-2}{k-3}+\binom{n-\ell-3}{k-3}$, by Lemma \ref{lem-key} we get
\begin{align}\label{ineq-thm3.5-3.10}
 f(\{1,\ell+1\})+2\ell f(\{\ell+1\})\leq f(\{1,2\})+\frac{\binom{n-\ell-3}{k-2}}{\binom{n-\ell-3}{k-3}}f(\{\ell+1\}) \leq \binom{n-\ell-1}{k-2}.
\end{align}
Since $\hf(\{1,2\},[\ell+1])$ and $\hf(\{3,\ell+1\},[\ell+1])$ are cross-intersecting, by Lemma \ref{lem-key} we get
\begin{align*}
 f(\{1,\ell+1\})+2\ell f(\{3,\ell+1\})\leq f(\{1,2\})+\frac{\binom{n-\ell-3}{k-2}}{\binom{n-\ell-3}{k-4}}f(\{3,\ell+1\}) \leq \binom{n-\ell-1}{k-2}.
\end{align*}
Similarly,
\begin{align*}
 f(\{1,\ell+1\})+2\ell f(\{2,\ell+1\})\leq f(\{1,3\})+2\ell f(\{2,\ell+1\})\leq \binom{n-\ell-1}{k-2}.
\end{align*}
In general, we have
\begin{align}\label{ineq-thm3.5-3.13}
  f(\{1,\ell+1\})+2\ell f(\{i,\ell+1\})\leq \binom{n-\ell-1}{k-2},\ i=2,3,\ldots,\ell.
\end{align}
Adding  \eqref{ineq-thm3.5-3.10}, \eqref{ineq-thm3.5-3.13} over $i=2,3,\ldots,\ell$ and dividing $2\ell$, we obtain that
\begin{align}\label{ineq-thm3.5-3.14}
f_1(\ell+1)=f(\{\ell+1\})+\sum_{1\leq i\leq \ell}f(\{i,\ell+1\}) \leq \frac{1}{2}\binom{n-\ell-1}{k-2}+\frac{1}{2} f(\{1,\ell+1\}).
\end{align}
For any $Q\subset [2,\ell+1]$ with $3\leq |Q|\leq \ell-1$, there exists some $x\in [2,\ell+1] \setminus Q$. Note that
$n\geq 2(\ell+1)k$ implies
\begin{align*}
\frac{\binom{n-\ell-3}{k-2}}{\binom{n-\ell-3}{k-|Q|-2}} &= \frac{(n-\ell-k+|Q|-1)\ldots (n-\ell-k-2)}{(k-2)\ldots (k-|Q|-1)}\\[3pt]
&>\left(\frac{n-\ell-k+|Q|-1}{k-2}\right)^{|Q|}\\[3pt]
&\geq (2\ell)^{|Q|}>2^{|Q|-1}\binom{\ell-1}{|Q|-1}.
\end{align*}
Since $\hf(\{1,x\},[\ell+1])$ and $\hf(Q,[\ell+1])$ are cross-intersecting, by Lemma \ref{lem-key} we get
\[
 f(\{1,\ell+1\})+2^{|Q|-1}\binom{\ell-1}{|Q|-1}f(Q)\leq f(\{1,x\})+\frac{\binom{n-\ell-3}{k-2}}{\binom{n-\ell-3}{k-|Q|-2}}f(Q)\leq \binom{n-\ell-1}{k-2}.
\]
It follows that
\[
f(Q)\leq  2^{-|Q|+1}\binom{\ell-1}{|Q|-1}^{-1}\left(\binom{n-\ell-1}{k-2}-f(\{1,\ell+1\})\right).
\]
Then
\begin{align}\label{ineq-thm3.5-3.15}
f_2(\ell+1,1)&=\sum_{i\in Q\subset [\ell+1],\ 1\notin Q,\ 3\leq |Q|\leq \ell-1} f(Q)\nonumber\\[3pt]
&=\sum_{i\in Q\subset [\ell+1],\ 1\notin Q,\ 3\leq |Q|\leq \ell-1} 2^{-|Q|+1}\binom{\ell-1}{|Q|-1}^{-1}\left(\binom{n-\ell-1}{k-2}-f(\{1,\ell+1\})\right) \nonumber\\[3pt]
&= \left(\binom{n-\ell-1}{k-2}-f(\{1,\ell+1\})\right) \sum_{3\leq j\leq \ell-1} 2^{-j+1}\nonumber\\[3pt]
&< \frac{1}{2}\binom{n-\ell-1}{k-2}-\frac{1}{2}f(\{1,\ell+1\}).
\end{align}
Adding \eqref{ineq-thm3.5-3.14} and \eqref{ineq-thm3.5-3.15}, we conclude that
\[
f_1(\ell+1)+f_2(\ell+1,1)<\binom{n-\ell-1}{k-2},
\]
proving \eqref{ineq-thm3.5-small2}.

\vspace{3pt}
{\bf Case 3.} $\hg=\emptyset$ and  $f(\{i\})\geq (\ell-1)\binom{n-\ell-3}{k-3}$ for $i=1,2,\ldots,\ell+1$.
\vspace{3pt}

Note that $\hf(\{i\},[\ell+1])$, $\hf(\{j\},[\ell+1])$ are cross-intersecting for $1\leq i<j\leq \ell+1$. By the Daykin Theorem, one of $\hf(\{i\},[\ell+1])$ and $\hf(\{j\},[\ell+1])$ has size at most $\binom{n-\ell-2}{k-2}$. Without loss of generality assume
\begin{align}\label{ineq-thm3.5-2}
|\hf(\{\ell+1\},[\ell+1])|\leq \binom{n-\ell-2}{k-2}.
\end{align}
Note that for any $Q\subsetneq [\ell+1]$ and $i\in [\ell+1]\setminus Q$, $\hf(\{i\},[\ell+1])$, $\hf(Q,[\ell+1])$ are cross-intersecting. Since $|\hf(\{i\},[\ell+1])|\geq (\ell-1)\binom{n-\ell-3}{k-3}$ for all $i=1,2,\ldots,\ell+1$, by Corollary  \ref{cor-hm}
\begin{align*}
|f(Q)|\leq  \binom{n-\ell-2}{k-|Q|-1}+\binom{n-2\ell-1}{k-|Q|-\ell+1}.
\end{align*}
It follows that
\begin{align}\label{ineq-thm3.5-3.8}
\sum_{\ell+1\in Q\subset [\ell+1],\ |Q|\geq 2} f(Q) \leq \sum_{2\leq j\leq \ell} \binom{\ell}{j-1} \left( \binom{n-\ell-2}{k-j-1}+\binom{n-2\ell-1}{k-j-\ell+1}\right)
\end{align}
Adding \eqref{ineq-thm3.5-2}, \eqref{ineq-thm3.5-3.8}  and $|\hf([\ell+1],[\ell+1])|\leq \binom{n-\ell-1}{k-\ell-1}$,  for $n\geq  4k$ we obtain that
\begin{align*}
f(\ell+1) &\leq \binom{n-\ell-2}{k-2}+\sum_{2\leq j\leq \ell} \binom{\ell}{j-1} \left( \binom{n-\ell-2}{k-j-1}+\binom{n-2\ell-1}{k-j-\ell+1}\right)+\binom{n-\ell-1}{k-\ell-1}.
\end{align*}
Using $\binom{n-2}{k-2} =\sum_{1\leq j\leq \ell+1} \binom{\ell}{j-1} \binom{n-\ell-2}{k-j-1}$, it follows that
\begin{align*}
f(\ell+1)&\leq   \binom{n-2}{k-2}+ \sum_{2\leq j\leq \ell} \binom{\ell}{j-1}\binom{n-2\ell-1}{k-j-\ell+1}+ \binom{n-\ell-2}{k-\ell-1}.
\end{align*}
Since for $2\leq j\leq \ell-1$ and $n\geq 2(\ell+1)k$
\[
\frac{\binom{\ell}{j}\binom{n-2\ell-1}{k-j-\ell}}{\binom{\ell}{j-1}\binom{n-2\ell-1}{k-j-\ell+1}} =\frac{(\ell-j+1)(k-j-\ell+1)}{j(n-\ell-k+j-1)} <\frac{\ell(k-\ell)}{n-\ell-k}< \frac{1}{2},
\]
we have
\[
\sum_{2\leq j\leq \ell} \binom{\ell}{j-1}\binom{n-2\ell-1}{k-j-\ell+1} <2 \ell\binom{n-2\ell-1}{k-\ell-1}<2 \ell\binom{n-\ell-2}{k-\ell-1}.
\]
Thus for $n\geq (2\ell+1)k$,
\begin{align*}
f(\ell+1)< \binom{n-2}{k-2}+ (2\ell+1)\binom{n-\ell-2}{k-\ell-1}<  \binom{n-2}{k-2}+ \binom{n-\ell-1}{k-\ell}
\end{align*}
and the theorem follows.

\vspace{3pt}
{\bf Case 4. } $\hg=\emptyset$ and $f(\{i\})< (\ell-1)\binom{n-\ell-3}{k-3}$ for some $i\in [\ell+1]$.
\vspace{3pt}

Without loss of generality assume that $f(\{\ell+1\})<(\ell-1)\binom{n-\ell-3}{k-3}$. By $\hg=\emptyset$, we infer that
\[
f(\{j,\ell+1\}) \leq \binom{n-\ell-2}{k-3}+\binom{n-\ell-3}{k-3},\ j=1,2,\ldots,\ell+1.
\]
Then for $n\geq 6\ell k$ we obtain that
\[
f_1(\ell+1) =\sum_{1\leq j\leq \ell} f(\{j,\ell+1\}) + f(\{\ell+1\}) < 3\ell\binom{n-\ell-2}{k-3} \leq \frac{1}{2}\binom{n-\ell-1}{k-2},
\]
proving \eqref{ineq-thm3.5-small}.
\end{proof}

\section{Concluding remarks}

In this paper, we mainly considered the maximum $i$th degree of an intersecting family for $2\leq i\leq  k+1$ and $i=2k+1$.

One of the main results states that every intersecting $k$-graph on $n$ vertices has at least $n-2k$ vertices of degree at most $\binom{n-2}{k-2}$ for $n\geq 6k$. Recall that the Huang-Zhao Theorem showed that  every intersecting $k$-graph on $n$ vertices has a vertex of degree at most $\binom{n-2}{k-2}$ for $n>2k$. An intriguing problem is whether our result holds for the full range.

\begin{conj}
Let $\hf\subset \binom{[n]}{k}$ be an intersecting family with $n>2k$. Then there are at least $n-2k$ vertices of degree at most  $\binom{n-2}{k-2}$.
\end{conj}

In Theorem \ref{thm-main-5} we proved that every $t$-intersecting $k$-graph on $n$ vertices has at least $n-k-1$ vertices of degree at most $\binom{n-t-1}{k-t-1}$ for $n\geq \binom{t+2}{2}k^2$. Let us close this paper by a stronger conjecture.

\begin{conj}
Let $\hf\subset \binom{[n]}{k}$ be a $t$-intersecting family. Then for
some absolute constant $c$ and $n\geq ckt$, $d_{k+2}(\hf)\leq \binom{n-t-1}{k-t-1}$.
\end{conj}

\vspace{5pt}
{\noindent \bf Acknowledgement.} The second author was supported by
National Natural Science Foundation of China Grant no. 12471316 and Natural Science Foundation of Shanxi Province Grant no. RD2500002993.


\begin{thebibliography}{10}


\bibitem{daykin}
D.E. Daykin, Erd\H{o}s-Ko-Rado from Kruskal-Katona, J. Combin. Theory, Ser. A 17 (1972), 254--255.

\bibitem{ekr}
P. Erd\H{o}s, C. Ko, R. Rado, Intersection theorems for systems of finite sets, Quart. J. Math. Oxford Ser. 12 (1961), 313--320.

\bibitem{F78} P. Frankl, The Erd\H{o}s-Ko-Rado theorem is true for $n = ckt$,  Coll. Math. Soc. J. Bolyai 18 (1978), 365--375.


\bibitem{F84}
P. Frankl, A new short proof for the Kruskal-Katona theorem, Discrete Math. 48  (1984), 327--329.

\bibitem{F87}
P. Frankl, The shifting technique in extremal set theory, Surveys in Combinatorics  123 (1987), 81--110.

\bibitem{F87-2}
P. Frankl, Erd\H{o}s-Ko-Rado theorem with conditions on the maximal degree, J. Comb. Theory, Ser. A 46(2) (1987), 252--263.





\bibitem{FW22} P. Frankl, J. Wang, Intersections and distinct intersections in cross-intersecting families, Europ. J. Combin.   110 (2022), 103665.



\bibitem{FW2024}
 P. Frankl, J. Wang, On the $C$-diversity of intersecting hypergraphs, Eur. J. Comb. 130 (2025), 104199.


\bibitem{FG}
 Z. F\"{u}redi, J. R. Griggs, Families of finite sets with minimum shadows, Combinatorica 6
 (1986), 355--363.

\bibitem{Hilton}
A.J.W. Hilton, The Erd\H{o}s-Ko-Rado Theorem with valency conditions, unpublished manuscript, 1976.


\bibitem{HM67}
 A.J.W. Hilton, E.C. Milner, Some intersection theorems for systems of finite sets, Quart.J. Math. Oxford Ser. 18 (1967), 369--384.

\bibitem{HZ}
 H. Huang, Y. Zhao, Degree versions of the Erd\H{o}s-Ko-Rado Theorem and Erd\H{o}s
 hypergraph matching conjecture, Journal of Combinatorial Theory, Ser. A, 150
 (2017), 233--247.

 \bibitem{Katona66}
G.O.H. Katona, A theorem of finite sets, Theory of Graphs. Proc. Colloq. Tihany, Akad. Kiad\'{o} (1966), 187--207.

 \bibitem{Kruskal}
J.B. Kruskal, The number of simplices in a complex, Mathematical Optimization Techniques 251 (1963), 251--278.


\bibitem{KZ2018}
A. Kupavskii, D. Zakharov, Regular bipartite graphs and intersecting families, J. Comb. Theory Ser. A 155 (2018), 180-189.





\bibitem{W84}
R. M. Wilson, The exact bound in the Erd\H{o}s-Ko-Rado theorem, Combinatorica 4 (1984), 247--257.

\end{thebibliography}
\end{document}